\documentclass[12pt,a4paper]{article}
\usepackage{graphicx,color}
\usepackage{amsmath,amsfonts,amssymb,amsthm}
\usepackage{hyperref} 
\usepackage{ulem}

\usepackage[numbers]{natbib}
\newcommand{\sr}[1]{{\cal #1}}
\newcommand{\dd}[1]{\mathbb{#1}}

\newcommand{\ol}{\overline}

\newcommand{\pend}{\hfill \thicklines \framebox(6.6,6.6)[l]{}}

\newenvironment{proof*}[1]{\noindent {\sc  #1} \rm}{\pend}

\newcommand{\eq}[1]{(\ref{eq:#1})}
\newcommand{\lem}[1]{Lemma~\ref{lem:#1}}
\newcommand{\cor}[1]{Corollary~\ref{cor:#1}}
\newcommand{\thr}[1]{Theorem~\ref{thr:#1}}

\newcommand{\rem}[1]{Remark~\ref{rem:#1}}

\newcommand{\app}[1]{Appendix~\ref{app:#1}}
\newcommand{\sectn}[1]{Section~\ref{sect:#1}}

\newcommand{\sect}[1]{\ref{sect:#1}}

\newcommand{\tN}{\widetilde{N}}
\newtheorem{theorem}{Theorem}[section]
\newtheorem{lemma}{Lemma}[section]

\newtheorem{proposition}{Proposition}[section]

\newtheorem{remark}{Remark}[section]

\newtheorem{corollary}{Corollary}[section]

\newcommand{\rd}{{\rm d}}
\newcommand{\re}{{\rm e}}
\newcommand{\bbE}{{\dd{E}}}
\newcommand{\bbP}{{\dd{P}}}
\newcommand{\bbR}{{\dd{R}}}
\newcommand{\bbZ}{{\dd{Z}}}
\newcommand{\barbbP}{{\overline{\bbP}}}

\newcommand{\wtillambda}{{\widetilde{\lambda}}}
\newcommand{\var}{\operatorname{var}}

\newcommand{\setsection}[2] {
\setcounter{section}{#1}
\setcounter{subsection}{0}
\setcounter{equation}{0}
\setcounter{conjecture}{0}
\setcounter{assumption}{0}
\setcounter{question}{0}
\setcounter{definition}{0}
\setcounter{theorem}{0}
\setcounter{corollary}{0}
\setcounter{lemma}{0}
\setcounter{proposition}{0}
\setcounter{remark}{0}
\setcounter{appen}{0}
\setsection*{\large \bf \thesection. #2}}

\newenvironment{mylist}[1]{\begin{list}{}
{\setlength{\itemindent}{#1mm}}
{\setlength{\itemsep}{0ex plus 0.2ex}}
{\setlength{\parsep}{0.5ex plus 0.2ex}}
{\setlength{\labelwidth}{10mm}}
}{\end{list}}

\newcommand{\setnewcounter} {
\setcounter{subsection}{0}
\setcounter{equation}{0}
\setcounter{conjecture}{0}
\setcounter{assumption}{0}
\setcounter{question}{0}
\setcounter{definition}{0}
\setcounter{theorem}{0}
\setcounter{corollary}{0}
\setcounter{lemma}{0}
\setcounter{proposition}{0}
\setcounter{remark}{0}
}

\begin{document}
\title{\bf \Large A martingale view of Blackwell's renewal theorem and its extensions to a general counting process}

\author{Daryl J. Daley\\The University of Melbourne \and Masakiyo Miyazawa
\\Tokyo University of Science}
\date{December 21, 2018} 

\maketitle

\begin{abstract}
Martingales constitute a basic tool in stochastic analysis; this paper
considers their application to counting processes. We use this tool to revisit
Blackwell's renewal theorem and its extensions for various counting processes.
We first consider a renewal process as a pilot example, deriving a new
semimartingale representation that differs from the standard decomposition
via the stochastic intensity function. We then revisit Blackwell's renewal
theorem, its refinements and extensions. Based on these observations,
we extend the semimartingale representation to a general counting process,
and give conditions under which asymptotic behaviour similar to Blackwell's
renewal theorem holds.
\end{abstract}
  
\begin{quotation}

\medskip

\end{quotation}

\section{Introduction}
\label{sect:introduction}
\setnewcounter

Let $N(t)$, $t \ge 0$, be a counting process with $\dd{E}[N(t)]$ finite for all $t \ge 0$.  We are interested in the asymptotic behaviour of
$\bbE[N(t)]$ for large $t$. For example, if $N(t)$ is a renewal process whose
lifetime distribution is nonarithmetic and has a finite mean $1/\lambda$,
Blackwell's renewal theorem, that for any finite $h>0$,
$\bbE[N(t+h)-N(t)] \to \lambda h$ for $t\to\infty$, is
a well known result, and has been extended to Markov renewal processes (see
e.g.\ \cite{Alsm1997,Cinl1975}). This theorem motivates us to consider under
what conditions it may hold for a general counting process.

To answer this question, we need a suitable description for the dynamics of a general counting process.
There have been various studies of refinements and extensions of Blackwell's renewal theorem (see e.g.\ \cite{Asmu2003,DaleMoha1978,DaleVere1988,Fell1971}), but they are based on independence or Markov assumptions on the counting process.
Hence, such traditional approaches may not be suitable for the present problem.
In this paper we use martingales to study this question.
In general, a martingale is used to construct an unbiased purely random component of a stochastic process.
For {\it any\/} counting process $N(t)$ that is assumed to be right-continuous
in $t$ and for which $\dd{E}[N(t)]<\infty$ for finite $t$, a martingale $M(t)$
typically arises in a semimartingale representation
\begin{equation}
\label{eq:decomposition 1}
  N(t) = \Lambda(t) + M(t), \qquad t \ge 0,
\end{equation}
where $\Lambda(t)$ is a process of bounded variation.
Here we must be careful about two things.
One is the filtration $\dd{F} \equiv \{\sr{F}_{t}; t \ge 0\}$ to which
$N(\cdot) \equiv \{N(t); t \ge 0\}$ is adapted; the other is the
predictability of $\Lambda(\cdot) \equiv \{\Lambda(t); t \ge 0\}$
with respect to the filtration, where $\Lambda(\cdot)$ is said to be $\dd{F}$-predictable (or simply predictable if $\dd{F}$ is clearly recognized) if for every $t \ge 0$, $\Lambda(t)$ is measurable with respect to
$\sr{F}_{t-} \equiv \sigma\big(\bigcup_{s < t} \sr{F}_s\big)$.

Consider the filtration $\dd{F}^{N} \equiv \{\sr{F}^{N}_{t}; t \ge 0\}$, where $\sr{F}^{N}_{t} = \sigma(\{N(u); u \le t\})$. When $\Lambda(\cdot)$ is $\dd{F}^{N}$-predictable, both it and therefore the martingale
$M(\cdot)$ are a.s.\ uniquely determined by virtue of the Doob--Meyer
decomposition because $N(\cdot)$ is a submartingale (see e.g.\ 
\cite[Lemma 25.7]{Kall2001}).
Such predictable $\Lambda(\cdot)$ is called a compensator, and
in this case $\Lambda(t)$ is nondecreasing in $t$.  Consequently,
if $\Lambda(t)$ is absolutely continuous with respect to Lebesgue measure,
then we can write 
\begin{equation}
\label{eq:compensator 1}
  \Lambda(t) = N(0) + \int_{0}^{t} \lambda_{u} \,\rd u,
\end{equation}
where $\lambda_{t}$ is a non-negative process and can be predictable with
respect to $\dd{F}^{N}$, and called a stochastic intensity.
In particular, $\lambda_{t}$ is called the {\it hazard rate} function when
$N(\cdot)$ is a renewal process (see e.g.\ \cite{Brem1981,JacoShir2003}). 
However, such $\lambda_{t}$ may not be amenable to our asymptotic analysis
except when             $\lambda_{t}$ is a deterministic function of $t$ or
its randomization, namely, $N(\cdot)$ is a Poisson or doubly stochastic
Poisson process, either of which is less
interesting for us. This may be the reason why the semimartingale
representation \eq{decomposition 1} has been little used in renewal theory.

In this paper, we study counting processes using martingales. However, we do
not use the filtration $\dd{F}^{N}$; rather, we consider non-predictable
$\Lambda(\cdot)$ that makes our asymptotic analysis tractable.
To this end, we formally introduce the counting process and related notations.
Let $N(\cdot)$ be a non-negative integer-valued process such that $N(t)$
is finite, nondecreasing, right-continuous, has left-hand limits in $t$ and
$\Delta N(t) \equiv N(t) - N(t-) \le 1$ for all $t \ge 0$, where $N(0-) = 0$.
Define the $n^{\rm th}$ counting time of $N(\cdot)$ by
\begin{align}
\label{eq:tndef}
  t_{n-1} = \inf\{t \ge 0; n \le N(t) \}, \qquad n \ge 1.
\end{align}
Since $N(t)$ is finite for all $t \ge 0$, $\{t_{n-1}, n=1,2,\ldots\}$ has no
accumulation point in $[0,\infty)$. Thus, $N(\cdot)$ is a general orderly
counting process. Unless stated otherwise, assume that $t_0 = 0$, and speak of
$N(\cdot)$ as being {\it non-delayed}.
Otherwise (so, $t_{0} > 0$), $N(\cdot)$ is said to be delayed. In either case,
$n \le N(t)$ if and only if $t_{n-1} \le t$. Let $T_0=t_0$ and
$T_n = t_n - t_{n-1}$ for $n \ge 1$.  Let $R(t)$ be
the residual time to the next counting instant at time $t$, that is,
\begin{align} \label{eq:Rdef}
  R(t) = T_0 + \sum_{\ell=1}^{N(t)} T_\ell - t, \qquad t \ge 0,
\end{align}
and define $X(t) = \big(N(t),R(t)\big)$ for $t \ge 0$, where
\begin{align*}
  X(0) = \begin{cases}
  (1,T_1) & \mbox{ if } t_0 = 0, \\
  (0,T_0) & \mbox{ if } t_0 > 0.
\end{cases}
\end{align*}

From the definition of $N(\cdot)$, for all $t\ge 0$ $X(t)$ is right-continuous
and has
a limit from the left. Hence, $N(t_{n}) = n+1$, $R(t_{n}-) = 0$
and $R(t_{n}) = T_{n+1}$.  Let $\dd{F}^{X} \equiv \{\sr{F}^{X}_{t}; t \ge 0\}$, where
\begin{equation}
\label{eq:filtration 1}
   \sr{F}^{X}_{t} = \sigma(\{X(u); u \le t\}), \qquad t \ge 0.
\end{equation}
Observe that $N(\cdot)$ is predictable under this filtration, but $R(\cdot)$ may not be the case because $R(t)$ can not be predicted by $\sr{F}^{X}_{t-}$ when $R(t-)=0$ unless all $T_{n}$s are deterministic.

In what follows, we use a filtration $\dd{F} \equiv \{\sr{F}_{t}; t \ge 0\}$ to which
$X(\cdot) \equiv \{X(t); t \ge 0\}$ is adapted, that is, $\sr{F}_{t}^{X} \subset \sr{F}_{t}$ for all $t \ge 0$, and we then write $\dd{F}^{X} \preceq \dd{F}$.
For convenience, we put $N(0-) = 0$ and $\sr{F}^{X}_{0-} = \sr{F}_{0-} = \sigma(R(0-))$ unless otherwise stated, where $R(0-) = R(0)1(N(0)=0)$. In most cases, $\dd{F} = \dd{F}^{X}$ is sufficient, but a larger $\dd{F}$ is needed in Sections \sect{modulated} and \sect{stationary IA}. For a stopping time $\tau$, define $\sr{F}_{\tau-}$ as
\begin{align*}
  \sr{F}_{\tau-} = \sigma(\{A\cap\{t < \tau\}; A\in \sr{F}_t , t \ge 0\}).
\end{align*}
Since $t_{n}$ is a stopping time with respect to $\dd{F}^{X}$, it is an $\dd{F}$-stopping time. Hence, we have the following fact.
\begin{lemma}[I.1.14 of \cite{JacoShir2003}]
\label{lem:measurable 1}
For each $n \ge 0$, $t_{n}$ is a stopping time and $\sr{F}_{t_{n}-}$-measurable.
\end{lemma}

We make the following two assumptions throughout the paper unless stated
otherwise.
\begin{itemize}
\item [(A.1)] $X(0) = (1,T_{1})$;
\item [(A.2)] $\dd{E}[T_{n}] < \infty$ for $n \ge 1$, where $T_{n} > 0$ almost surely by the orderliness of $N(\cdot)$.
\end{itemize}

First we consider a renewal process as a pilot example. In this case, we
assume the following
in addition to (A.1) and (A.2).
\begin{itemize}
\item [(A.3)] For $n \ge 1$, $T_{n}$ is independent of $\sr{F}_{t_{n-1}-}$,
and $T_1, T_2, \ldots$ are identically distributed; their common
distribution is denoted by $F$.  If $\dd{F} = \dd{F}^{X}$, then these
conditions are the same as $T_1, T_2, \ldots$ being independent and identically
distributed with common distribution $F$.
\end{itemize}
For such a renewal process, we derive a new semimartingale representation; it
differs from the standard representation that uses the stochastic intensity
function (see \thr{martingale 2}). This enables us to revisit Blackwell's
renewal theorem and consider some of its refinements and extensions.
Based on these observations, we extend the semimartingale representation
to a \textit{general counting process}, for which we replace the key renewal
assumption (A.3) by
\begin{itemize}
\item [(A.4)] $\dd{E}[N(t)] < \infty$ for finite $t \ge 0$.
\end{itemize}
We then give conditions under which an asymptotic result similar to Blackwell's
renewal theorem holds. In particular, Blackwell's result extends quite
naturally to the counting process which is generated by a stationary sequence
of inter-arrival times, in other words, the counting process under a Palm
distribution (see \cor{Palm 1}).

Our paper has five more sections.  In \sectn{martingale}, we present a general
framework for the new semimartingale representation.  We apply it to a renewal
process $N(t)$, and consider its interpretation.
In \sectn{revisit}, we revisit Blackwell's renewal theorem and some other
asymptotic properties of $\dd{E}[N(t)]$.  In \sectn{2nd},
we show how the present approach can be used to derive limit properties of
$\var N(t)$. In \sectn{extension}, a new semimartingale
representation is derived for a general counting process, and Blackwell's
renewal theorem is extended to a range of scenarios. We give concluding
remarks in \sectn{concluding}.
Some proofs are deferred to an appendix. 


\section{A semimartingale representation}
\label{sect:martingale}
\setnewcounter

As discussed in \sectn{introduction}, our aim is to derive the semimartingale
representation \eq{decomposition 1} for a general counting process $N(\cdot)$.
We first consider this problem in a broader context. 

\subsection{Martingales from a general setting}
\label{sect:deriving}

Let $\dd{F} \equiv \{\sr{F}_{t}; t \ge 0\}$ be a filtration, let $N(\cdot)$
with $N(0-)=0$ be the orderly counting process introduced in
\sectn{introduction}, and let $Y(\cdot) \equiv \{Y(t); t \ge 0\}$ with
$Y(0-) \in \sr{F}_{0-}$ be a real-valued stochastic process that is
right-continuous and has left-hand limits.  Then assume that
\begin{itemize}
\item [(B.1)] $N(\cdot)$ is $\dd{F}$-predictable, and $Y(\cdot)$ is
$\dd{F}$-adapted;
\item [(B.2)] for every $t\ge 0$, $N(t)$ increases if $Y(t) \ne Y(t-)$; and
\item [(B.3)] for every $t\ge 0$, $Y(t)$ has a right-hand
derivative, denoted $Y'(t)$.
\end{itemize}
Note that (B.2) does not exclude the case that $N(t)$ increases when
$Y(t) = Y(t-)$. Further, define $R(t)$ by \eq{Rdef}, and let $X(t) =
\big(N(t),R(t))$; then it is not hard to see that (B.1) implies that
$\dd{F}^X \preceq \dd{F}$.

Since $t_0 =0$ only if $N(0)=1$, it now follows easily from elementary
calculus that
\begin{equation}
\label{eq:evolution 1}
  Y(t) = Y(0-) + \int_0^t Y'(u)\,\rd u +\sum_{n=0}^{N(t)-1}\Delta Y(t_{n}),
\qquad t \ge 0,
\end{equation}
where $\Delta Y(t) = Y(t) - Y(t-)$. Let
\begin{align}
\label{eq:martingale 1}
 & M_{Y}(t) := \sum_{n=0}^{N(t)-1} \big(Y(t_n) - \bbE[Y(t_n)\mid\sr{F}_{t_n-}]\big),\\
\label{eq:DY 1}
 & D_Y(t) :=
\sum_{n=0}^{N(t)-1} \big(\bbE[Y(t_n)\mid\sr{F}_{t_{n}-}] - Y(t_n-)\big).
\end{align}
With these two functions we then have the following lemma.

\begin{lemma}
\label{lem:martingale 1}
Assume that $N(\cdot) \equiv \{N(t); t \ge 0\}$ with $N(0-)=0$ and $Y(\cdot)$
with $Y(0-) \in\sr{F}_{0-}$ satisfy conditions {\rm(A.4)} and {\rm(B.1)--(B.3)}.
If
\begin{mylist}{0}
\item [\rm (B.4)] for some constant $c_{0} > 0$, $\sup_{n \ge 0}
 \dd{E}\big[|\Delta Y(t_{n})|\bigm| \sr{F}_{t_{n}-}\big] \le c_{0}$ almost surely,
\end{mylist}
then $M_Y(\cdot)$
is an $\dd{F}$-martingale, and
\begin{equation}
\label{eq:decomposition 2}
  Y(t) = Y(0-) + \int_0^t Y'(u) \,\rd u + D_{Y}(t) + M_{Y}(t),  \qquad t \ge 0.
\end{equation}
\end{lemma}

\begin{remark} \rm
\label{rem:martingale 1a}
Condition (B.4) may look too strong, but it works perfectly in this paper. If
this condition fails, we may consider a local martingale by a sequence of
stopping times $\tau_{k} \equiv \inf\{t_{n} \ge 0; |\Delta Y(t_{n})| > k\}$
for $k \ge 1$. Then, similarly to the proof below, we can show that
$M_Y(\cdot)$ is a local $\dd{F}$-martingale.
\end{remark}

\begin{remark} \rm
\label{rem:martingale 1b}
Since in \eq{decomposition 2} both the integration term and $D_Y(t)$ are
predictable, the representation \eq{decomposition 2} for $Y(t)$ is a special
semimartingale, where ``special'' means that the bounded
variation component of a semimartingale is
predictable (see \cite[Chapter 1, section 4c]{JacoShir2003}).
\end{remark}

\begin{proof}
Equation \eq{decomposition 2} follows immediately from \eq{evolution 1}.  Thus,
we only need to prove that $M_{Y}(\cdot)$ is an $\dd{F}$-martingale. For this, note that $n + 1 \le N(t)$ if and only if
$t_n \le t$. Since $Y(t_n) = Y_{t_{n}-} + \Delta Y(t_n)$, (A.4) and (B.4)
imply that 
\begin{align*}
 & \dd{E}[|M_Y(t)|] \le \dd{E} \bigg[\sum_{n=0}^\infty \big|\big(Y(t_n) - \bbE[Y(t_n)\mid\sr{F}_{t_n-}]\big)
\,1(t_n \le t) \big| \bigg]\nonumber\\
 & \quad \le \sum_{n=0}^\infty \dd{E} \Big[ \big(|\Delta Y(t_n)| + 
  \bbE\big[|\Delta Y(t_n)|\mid\sr{F}_{t_n-}\big]\big) 1(t_n \le t\big) \Big] \le 2c_{0} \dd{E}[N(t)] < \infty,
\end{align*}
where $1(\cdot)$ is the indicator function of the statement ``$\cdot$''. Hence, we only need to show that 
\begin{equation}
\label{eq:martingale condition 1}
 \bbE[M_Y(t)\mid\sr{F}_{s}] = M_Y(s), \qquad 0 \le s < t.
\end{equation}
To this end, recall that $t_n$ is $\sr{F}_{t_n-}$-measurable by
\lem{measurable 1}. Hence, we have
\begin{align*}
 M_Y(t) & = \sum_{n=0}^\infty \big(Y(t_n) - \bbE[Y(t_n)\mid\sr{F}_{t_n-}]\big)
\,1(t_n \le t)\nonumber\\
 & = \sum_{n=0}^\infty \Big(\Delta Y(t_n) \,1(t_n \le t) 
  - \bbE[\Delta Y(t_n) \,1(t_n \le t)\mid\sr{F}_{t_n-}]\Big).
\end{align*}
This implies that $\bbE[M_Y(t)\mid\sr{F}_{s}] - M_{Y}(s)$ equals
\begin{align*}
 & \bbE \left(\sum_{n=0}^\infty \big(\big[\Delta Y(t_n)\, 1(s < t_n \le t)
 - \bbE\big[\Delta Y(t_n) 1(s<t_n \le t)\bigm|\sr{F}_{t_n-}\big]\,\big)\Biggm|\sr{F}_s \right)\\
 & \quad = \sum_{n=0}^\infty \bbE\big(\Delta Y(t_n)\, 1(s < t_n \le t)
 - \bbE\big[\Delta Y(t_n) 1(s<t_n \le t)\bigm|\sr{F}_{t_n-}\big]\, \bigm| \sr{F}_s\big)
  = 0,
\end{align*}
where the interchange of the expectation and the summation is justified by
(A.4) and (B.4). This proves \eq{martingale condition 1}, and therefore
$M_Y(\cdot)$ is an $\dd{F}$-martingale.
\end{proof}

In what follows, we apply Lemma \ref{lem:martingale 1} with appropriately
chosen functions $Y(\cdot)$.  For example, we may put $Y(t) = N(t)$ for a
$\dd{F}$-predictable counting process $N(\cdot)$ with $N(0-)=0$ because
(B.1)--(B.3) are obviously satisfied. Then
 $D_Y(t) \equiv N(t)$ and $M_Y(t) \equiv 0$,
and substitution in \eq{decomposition 2} yields the identity $Y(t) = N(t)$, 
so we should learn nothing.  Thus,
it is important to choose $Y(\cdot)$ suitably when applying \lem{martingale 1}.

\subsection{Application to a renewal process}
\label{sect:renewal}

In our first application of \lem{martingale 1} we find the semimartingale
representation \eq{decomposition 1} for a renewal process $N(\cdot)$ defined
by (A.1)--(A.3). This representation is used in establishing asymptotic
properties of moments of $N(t)$ including Blackwell's renewal theorem in
\sectn{revisit}, and extended to a general counting process in
\sectn{extension}.
We first note the following well known fact (see e.g.\
\cite[Lemma in XI.1]{Fell1971}).
\begin{lemma}
\label{lem:N(t) finite}
Conditions {\rm(A.1)--(A.3)} imply {\rm(A.4)}.
\end{lemma}

We choose a filtration $\dd{F}$ such that $\dd{F}^{X} \preceq \dd{F}$, where
$\dd{F}^{X}$ is defined through \eq{filtration 1}. Let $Y(t) =
N(t)-\lambda R(t)$; such $Y(\cdot)$ obviously satisfies conditions
(B.1)--(B.3). The idea behind this choice of $Y(\cdot)$ is to introduce a
control parameter $\lambda$ so that $D_Y(t)$ vanishes. A similar idea is used
for the queue length process of a many-server queue in \cite{Miya2017}. Indeed,
\begin{align}
\label{eq:DY 2}
  \bbE[Y(t_n)\mid\sr{F}_{t_n-}] = N(t_n-) + 1 - \lambda \bbE[T_{n+1}]
\,=\, N(t_n-) \,=\, Y(t_n-), \qquad n \ge 0,
\end{align}
and therefore $D_Y(t_n)$ vanishes, where $N(0-) = 0$ and we recall that
$\sr{F}_{0-} = \{\emptyset, \Omega\}$. Further, (A.4) holds
by \lem{N(t) finite}, and we have (B.4) from the bound
\begin{align*}
  \bbE[|\Delta Y(t_{n})|\mid\sr{F}_{t_n-}] =
\bbE\big[|1-\lambda T_{n+1}|\big] \le 2 < \infty.
\end{align*}
Hence, the next theorem is immediate from \lem{martingale 1} because $Y(0-) = 0$ by $R(0-)=0$.

\begin{theorem}
\label{thr:martingale 2}
Let $\dd{F}$ be a filtration such that $\dd{F}^{X} \preceq \dd{F}$, and assume {\rm(A.1)--(A.3)}. Then the renewal process $N(t)$ is expressible as
\begin{equation}
\label{eq:decomposition 3a}
   N(t) = \lambda \big(t+R(t)\big) + M(t), \qquad t \ge 0,
\end{equation}
where
\begin{equation}
\label{eq:martingale 2}
  M(t) = \sum_{n=0}^{N(t)-1} (1 - \lambda T_{n+1})
\,= \sum_{n=1}^{N(t)} (1 - \lambda T_{n}), \qquad t \ge 0,
\end{equation}
is an $\dd{F}$-martingale.
\end{theorem}

\begin{remark}
\label{rem:martingale 2}
{\rm $N(\cdot)$ is called a delayed renewal process when condition (A.1) is
replaced by $X(0) = (0,T_{0})$ with $T_{0} > 0$, while (A.2) and (A.3) are
unchanged. Let $Y(t) = N(t) - \lambda R(t)$ for $t \ge 0$ and $Y(0-) = R(0)$.
Then $Y(0-) = -\lambda T_0$, and
\begin{align*}
  & D_Y(t) = 0, \qquad M_Y(t) = M(t), \qquad t \ge 0.
\end{align*}
Hence, by \lem{martingale 1},
\eq{decomposition 3a} for the delayed renewal process becomes 
\begin{equation}
\label{eq:decomposition 3b}
   N(t) = \lambda \big(t+R(t) - T_0\big) + M(t), \qquad t \ge 0.
\end{equation}
In particular, if $R(t)$ is stationary, then $N(t)$ has stationary increments because
\begin{align*}
  N(t) = \sum_{0 < u \le t} 1\big(R(u-) \ne R(u)\big).
\end{align*}
Hence, \eq{decomposition 3b} and $T_0 = R(0)$ imply that $M(t)$ also has
stationary increments.
}\end{remark}  

This remark shows that the delayed renewal process only changes the
semimartingale representation \eq{decomposition 3a} by having the extra term
$-\lambda T_0$ as in \eq{decomposition 3b}. Since $T_0$ is independent of all
$T_n$s for $n \ge 1$ by (A.3), and the weak convergence of $R(t)$ to its
stationary distribution as $t \to \infty$ is a key step in our arguments, the
asymptotic results in Sections \sect{revisit} and \sect{2nd} are also valid for
the delayed renewal process if $T_0$ has an appropriate finite moment. Since
the extensions are then obvious, we do not discuss them in what follows.

\subsection{Interpretation of the semimartingale representation}
\label{sect:interpretation}
By \thr{martingale 2}, we have the semimartingale representation
\eq{decomposition 1} with
\begin{align*}
  \Lambda(t) = \lambda [t+R(t)],
\end{align*}
which differs from the compensator \eq{compensator 1}. This is not
surprising because we have made a special semimartingale not for $N(t)$ but for
$Y(t) \equiv N(t)-\lambda R(t)$ (recall the special martingale
discussed in \rem{martingale 1b}). Further, the filtration is different, and
$\Lambda(\cdot)$ is not predictable because $R(\cdot)$ need not be predictable. 
Nevertheless, \eq{decomposition 1} suggests that the asymptotics of $N(t)$
can be studied via a bias term $\lambda[t+R(t)]$ and a pure noise term $M(t)$.

Another feature of \eq{decomposition 3a} is its relation to Wald's identity.
Define $S_n = \sum_{\ell=1}^n T_{\ell}$ for $n \ge 1;$  then
$S_{N(t)} = t+R(t)$, and therefore \eq{decomposition 3a} can be written as
\begin{align}
\label{eq:Wald 1}
  S_{N(t)} - \bbE[T] N(t) = - \bbE[T] M(t), \qquad t \ge 0,
\end{align}
which immediately leads to Wald's identity, $\bbE[S_{N(t)}]=\bbE[T]\bbE[N(t)]$,
since $\bbE[M(t)] = \bbE[M(0)] = 0$.  This type of Wald's
identity is well known (see e.g.\ \cite[\S{V.6}]{Asmu2003}).

What is interesting here is that \eq{Wald 1} says more.
For example, the $\dd{F}$-martingale $-\bbE(T) M(t)$ is an error for
estimating $S_{N(t)}$ by $[\bbE(T)]N(t)$.
To evaluate this error, we use certain facts concerning the quadratic
variations of $M(\cdot)$ (see \cite{Vaar2006} for their definitions).  
\begin{lemma}
\label{lem:M2 1}
Under the assumptions of \thr{martingale 2}, the
optional quadratic variation of $M(\cdot)$ is given by
\begin{equation}
\label{eq:M2 1}
  [M](t) = \sum_{n=1}^{N(t)} (1 - \lambda T_{n})^{2},
\end{equation}
and, if\/ $\bbE(T^2) < \infty$, the predictable quadratic variation of
$M(\cdot)$ is given by
\begin{align}
\label{eq:M2 2}
 & \langle M \rangle(t) = \lambda^{2} \sigma_T^2 N(t), \qquad t \ge 0,
\end{align}
where $\sigma_{T}^{2}$ is the variance of $T$, and therefore
\begin{align}
\label{eq:E M2 2}
  \bbE[M^2(t)] = \bbE[\langle M \rangle(t)]
 \,=\, \lambda^3 \sigma_T^2 (t + \bbE[R(t)]).
\end{align}
\end{lemma}
\begin{proof}
Since $M(t)$ is piecewise constant and discontinuous at increasing instants of
$N(t)$, \eq{M2 1} is immediate from the definition of an optional quadratic
variation (see e.g.\ \cite[Theorem 3.1]{PangTalrWhit2007}). Since the
predictable quadratic variation $\langle M\rangle(t)$ is defined as a
predictable process for $M^2(t) - \langle M\rangle(t)$ to be a martingale,
\eq{M2 2} is obtained from $\bbE(T^2) < \infty$ and \lem{martingale 1} on setting $Y(t) = M^2(t)$.
Its proof is detailed in \app{M2 2}. Finally we obtain \eq{E M2 2} from
\eq{M2 2} and (2.7). 
\end{proof}

It is notable that $N(\cdot)$ is predictable but $T_{N(\cdot)}$ is not in our
filtration $\dd{F}$, while neither of them is predictable in the filtration
$\dd{F}^{N}$ generated by $N(\cdot)$. This explains why $[M](t)$ of \eq{M2 1}
differs from $\langle M\rangle(t)$ of \eq{M2 2}.

If $\bbE[T^2] < \infty$, it follows from Lemma \ref{lem:M2 1} and the
inequality $\bbE[R(t)] \le \lambda \bbE[T^2]$ for $t \ge 0$ (see e.g.\
\cite[Proposition 6.2, p.160]{Asmu2003}) that the expected quadratic error of
\eq{Wald 1} is
\begin{equation}
\label{eq:error bound 1}
  \big(\bbE[T]\big)^2\bbE[M^2(t)] = \sigma_T^2 \lambda\big(t+\bbE[R(t)]\big)
 \le \sigma_T^2 \big(\lambda t + \lambda^2 \bbE[T^2] \big).
\end{equation}

\section{First moment asymptotics for a renewal process}
\label{sect:revisit}
\setnewcounter  
We have asserted that the semimartingale representation \eq{decomposition 3a}
can be used to find the asymptotics of a renewal process $N(t)$ for large $t$.
In this subsection, we consider them for the first moment under two scenarios
depending on the finiteness or otherwise of $\bbE[T^2]$.

\subsection{Blackwell's renewal theorem, revisited}
\label{sect:Blackwell}
The first moment asymptotic is well known as Blackwell's renewal theorem.
In view of the representation \eq{decomposition 3a}, the asymptotic behaviour
of $\bbE[N(t)]$ is determined by that of $\bbE[R(t)]$.  Taking
this into account, we reformulate Blackwell's renewal theorem as follows.

\begin{lemma}
\label{lem:Renewal 1}
For a renewal process $N(\cdot)$ satisfying assumptions {\rm (A.1)--(A.3)},
the following three conditions are equivalent.
\begin{itemize}
\item [(\sect{martingale}a)] The distribution of $T$ is non-arithmetic, i.e.\
there is no $\delta > 0$ such that $\dd{P}(T \in \{n \delta; n \ge 1\}) = 1$.
\item [(\sect{martingale}b)] Blackwell's renewal theorem holds, i.e.\ for each
$h > 0$ and $\lambda = 1/\bbE[T]$,
\begin{equation}
\label{eq:Blackwell 1}
  \lim_{t \to \infty} \bbE[N(t+h) - N(t)] = \lambda h.
\end{equation}
\item [(\sect{martingale}c)] $R(t)$ has a limiting distribution as
$t \to \infty$.
\end{itemize}
When one of these conditions holds, the limiting distribution of $R(t)$ is
given by
\begin{align}
\label{eq:R limit 1}
 \lim_{t\to\infty} \dd{P}\big(R(t) \le x\big) = \lambda \int_0^x \dd{P}(T > u)
\,\rd u.
\end{align}
\end{lemma}

\begin{remark}
\label{rem:Renewal 1}
{\rm Following equation \eq{residual 1} below, there is a direct
proof that (\sect{martingale}c) implies (\sect{martingale}b), and hence of
their equivalence in the delayed case when the non-arithmetic condition in
(\sect{martingale}a) may not be necessary for $R(t)$ to have a limiting
distribution in (\sect{martingale}c).
}\end{remark}

\begin{proof}
This lemma owes its proof to Feller's key renewal theorem  \cite[Chapter XI]{Fell1971}.
By Blackwell's renewal theorem, which is a special case of Feller's key renewal
theorem, (\sect{martingale}a) implies (\sect{martingale}b).  By Feller's direct
Riemann integrability argument, (\sect{martingale}b) implies
(\sect{martingale}c) and \eq{R limit 1} (see \cite[\S{XI.4}]{Fell1971}).
Finally, (\sect{martingale}c) implies (\sect{martingale}a) because
(\sect{martingale}c) does not hold if (\sect{martingale}a) does not hold,
equivalently, $F$ is arithmetic.
\end{proof}

Up to this point, the asymptotic behaviour of $\bbE[N(t)]$ provided by
\lem{Renewal 1} has nothing to do with the semimartingale representation
\eq{decomposition 3a}.  However, when we look at the problem from a sample
path viewpoint, \eq{decomposition 3a} can be considered as a pre-limit renewal
theorem. Taking its expectation, we have
\begin{equation}
\label{eq:identity 1}
  \bbE[N(t)] = \lambda t + \lambda \bbE[R(t)], \qquad t \ge 0,
\end{equation}
which is equivalent to Wald's identity as discussed in \sectn{interpretation}. 
It is of interest here to see how Blackwell's renewal theorem \eq{Blackwell 1}
can be obtained directly from \eq{identity 1} which provides information on
$R(t)$, namely, \eq{Blackwell 1} holds if and only if
\begin{equation}
\label{eq:residual 1}
  \lim_{t \to \infty} \bbE[R(t+h) - R(t)] = 0.
\end{equation}
If $\bbE[T^2] < \infty$, then \eq{residual 1} is immediate from
(\sect{martingale}c) because $\bbE[R(t)]$ converges to
$\frac12\lambda\bbE(T^2)$ as $t\to\infty$. However, this argument does not
apply when $\bbE[T^2] = \infty$. Nevertheless, \eq{residual 1}, equivalently,
(\sect{martingale}b), is still obtained directly from (\sect{martingale}c)
using another semimartingale representation of $N(t)$ as we now show.

\begin{proof}[Direct proof of (\sect{martingale}b) and \eq{R limit 1} from
(\sect{martingale}c)] Let $\widetilde{R}$ be a random variable distributed the
same as the limit distribution of $R(t)$.  Apply \lem{martingale 1} with 
$Y(t)=N(t)-\lambda_v \big(v \wedge R(t)\big)$ for each fixed
$v \in C_{\widetilde{R}}$, where $a\wedge b =\min(a,b)$ for $a,b\in\bbR$,
$\lambda_v = 1/\bbE[v\wedge T]$, and $C_{\widetilde{R}}$ is the set of all
continuity points of the distribution of $\widetilde{R}$.  Similarly to
\eq{DY 2}, we can check that $D_Y(t) = 0$
and $Y'(t) = \lambda_v 1\big(R(t) < v\big)$. Hence,
\begin{equation*}
  M_v(t) := \sum_{n=1}^{N(t)} \big(1 - \lambda_v (v \wedge T_n)\big),
\end{equation*}
is an $\dd{F}$-martingale for the filtration $\dd{F}$ such that $\dd{F}^{X} \preceq \dd{F}$, and, for $v > 0$, 
\begin{equation}
\label{eq:truncated 3}
  N(t) = \lambda_v \int_0^t 1\big(R(u)<v\big)\,\rd u 
  +\lambda_v \big(v\wedge R(t)\big)+M_v(t).
\end{equation}
Taking expectations in \eq{truncated 3} yields
\begin{equation}
\label{eq:truncated 1}
  \bbE[N(t)] = \lambda_v \int_0^t \bbP\big(R(u) < v\big) \,\rd u 
  + \lambda_v \bbE[v \wedge R(t)], \qquad t \ge 0,
\end{equation}
and therefore, if (\sect{martingale}c) holds, then, as $u \to \infty$,
$\bbP\big(R(u) < v\big)$ converges to $\bbP(\widetilde{R} < v)$, which equals
$\bbP(\widetilde{R} \le v)$ at continuity points
$v \in C_{\widetilde{R}}$, so for such $v$,
\begin{equation}
\label{eq:lambda 0}
  \lim_{t\to\infty} \bbE[N(t)]\big/t = \lambda_v \bbP(\widetilde{R} \le v).
\end{equation}
Since the left-hand side of this equation is independent of $v$, we have
$\lambda_{v} \bbP(\widetilde{R} \le v) = \lambda^{*}$ for some $\lambda^{*}$
and for all $v \in C_{\widetilde{R}}$. Hence, we have
\begin{align}
\label{eq:limiting R 1}
  \bbP(\widetilde{R} \le v) = \frac {\lambda^{*}} {\lambda_v}
  = \lambda^{*} \bbE[v \wedge T] 
  = \lambda^{*} \bbE\left[\int_0^v 1(T > x)\,\rd x\right],
\end{align}
and this equality holds for all $v \ge 0$ because the right-hand side is
continuous in $v$.  Letting $v\to\infty$ in \eq{limiting R 1}, shows that
$1 = \lambda^{*} \bbE[T]$. Hence, $\lambda^{*} = 1/\bbE[T] = \lambda$.
This and \eq{limiting R 1} yield \eq{R limit 1}. It follows from
\eq{truncated 1} and \eq{limiting R 1} that, for each $h > 0$ and
$v \in C_{\widetilde{R}}$,
\begin{equation}
\label{eq:truncated 2}
  \lim_{t\to\infty} \big(\bbE[N(t+h)] - \bbE[N(t)]\big) 
  = \lambda_v \lim_{t\to\infty}
    \int_t^{t+h} \bbP\big(R(u) < v\big)\,\rd u = \lambda h,
\end{equation}
because $v \wedge R(t)$ is bounded by $v$. Thus, (\sect{martingale}c) implies
(\sect{martingale}b) and \eq{R limit 1}.
\end{proof}

The truncation technique used in this proof is useful for more general counting
processes as we show in \sectn{extension}.

\subsection{Infinite second moment case}
\label{sect:infinite}

When $\bbE[T^2] = \infty$, it is of interest to consider a refinement of the
elementary renewal theorem $\bbE[N(t)] - \lambda t = o(t)$.  Sgibnev
\cite{Sgib1981} studied this problem, starting with the case of an
arithmetic lifetime distribution. Here we consider it through the asymptotic
behaviour of $\bbE[R(t)]$ in \eq{identity 1}.
Recall first that, for some function $z(\cdot): \bbR_+ \mapsto \bbR$, called a
generator of $Z(\cdot)$, a solution $Z$ of the general renewal equation of
\citet{Fell1971},
\begin{align}
\label{eq:Z 1}
  Z(t) = z(t) + \int_0^t Z(t-u) \,F(\rd u), \qquad t \ge 0,
\end{align}
is given by
\begin{equation*}
  Z(t) = \bbE \bigg[ \int_{0}^{t} z(t-u) \,N(\rd u) \bigg].
\end{equation*}
We exhibit $\bbE[R(t)]$ as a solution of the general renewal equation.
From \eq{Rdef}, when $t_{n-1} \le t < t_n$,
$R(t) = T_n - (t-t_{n-1})$ so choosing
\begin{align}
\label{eq:z 1}
  z(t) &= \bbE\big[(T - t)\,1(T>t)\big],
\end{align}
and noting the fact that $t_{n-1}$ and $T_{n}$ are independent, we have
\begin{align}
\label{eq:gre R 1}
  \bbE[R(t)] & = \bbE\bigg[ \sum_{n=1}^\infty \big(T_n - (t-t_{n-1})\big)
  \,1(t_{n-1} \le t < t_{n-1} + T_n) \bigg] \nonumber\\
 & = \bbE\bigg[\sum_{n=1}^\infty \dd{E}\big[\big(T_n - (t-t_{n-1})\big)
1(t_{n-1} \le t < t_{n-1} + T_n)\,\bigm|\, t_{n-1} \big] \bigg] \nonumber\\
 &= \bbE\bigg[ \int_0^t z(t-u) \,N(\rd u) \bigg].
\end{align}
Thus, $z(\cdot)$ is indeed a generator for $\bbE[R(t)]$. To check the
asymptotic behaviour of \eq{gre R 1}, the following lemma is useful (see also
\cite{Sgib1981} or \cite[Exercise 4.4.5(c)]{DaleVere1988}); in the lemma and
elsewhere, $f(t) \sim g(t)$ for functions $f, g : \bbR_+ \mapsto \bbR$ means
$\lim_{t \to \infty} f(t)/g(t) = 1$.

\begin{lemma}
\label{lem:Sgib1981}
{\rm(Sgibnev \cite[Theorem 4]{Sgib1981})}
If in $\eq{Z 1}$ the generator $z(t)$ is non-negative and non-increasing in
$t \ge 0$, then the solution $Z(\cdot)$ satisfies
\begin{equation}
\label{eq:Z 2}
  Z(t) \sim \lambda \int_0^t z(u) \,\rd u.
\end{equation}
\end{lemma}

It follows from \thr{martingale 2} that \eq{z 1}, \eq{gre R 1}
and \lem{Sgib1981} yield
\begin{align}
\label{eq:gre R 2}
  \bbE[N(t)] - \lambda t  = \lambda \bbE[R(t)] 
  & \sim \lambda \int_0^t \int_u^\infty \bbP(T>x) \,\rd x \,\rd u
\end{align}
as shown in \cite[Theorem 5]{Sgib1981}. The asymptotic behaviour of
\eq{gre R 2} may be viewed as a doubly integrated tail of the distribution $F$
of $T$ (see e.g.\ \cite{FossKorsZach2011} for an integrated tail).

In this section, we have observed how the semimartingale representations
\eq{decomposition 3a} and \eq{truncated 3} are helpful in elucidating the
asymptotic behaviour of $\dd{E}[N(t)]$. One may wonder how the present approach
might work for the asymptotic behaviour of higher moments of $N(t)$.
This is considered in \sectn{2nd}. 

It is also of interest to see how the approach works for more general
counting processes. Observe that \eq{decomposition 3a} holds if $D_{Y}(t)$ of
\eq{DY 1} vanishes, for which $N(t)$ need not necessarily be a renewal
process --- we discuss this extension in \sectn{extension} where the
exposition is independent of the results in \sectn{2nd}.

\section{Second moment asymptotics}
\label{sect:2nd}
\setnewcounter

We consider the variance of the renewal process $N(t)$, denoted $\var N(t)$.
As shown below, the representation of Theorem 2.1
gives an alternative path for studying the
asymptotic behaviour of $\var N(t)$. In particular, the martingale $M(t)$ plays
an important role in this case, and this contrasts with the first moment case.

Begin by using \eq{decomposition 3a} with $\bbE[M(t)] = 0$ to compute
$\var N(t)$ in the form
\begin{equation}
\label{eq:variance 1}
\var N(t) = \lambda^2 \var R(t) + 2 \lambda \bbE[R(t) M(t)] + \bbE[M^2(t)].
\end{equation}
From \lem{M2 1} we know that when $\bbE[T^2]$ is finite, 
$\bbE[M^2(t)] \sim \lambda^3 \sigma^2_T t$.  We therefore assume that
$\bbE[T^2] < \infty$ because otherwise $\var N(t)$ is not finite.  To study the
asymptotic behaviour of $\var R(t)$, we consider 
$\bbE[R^2(t)]$; this function is the solution of the general renewal equation
\eq{Z 1} with the generator (cf.\ around \eq{z 1} above)
\begin{equation}
\label{eq:z 2}
  z(t) = z_2(t) := \bbE[(T-t)^2\,1(T>t)] = \int_t^\infty 2x\,\bbP(T>x)\,\rd x.
\end{equation}
Let
\begin{equation}
\label{eq:h 2}
  h(t) = \int_0^t z_2(u) \,\rd u
 = t \bbE[(T-t)T\,1(T>t)] + \tfrac 13 \bbE[(T\wedge t)^3].
\end{equation}
Then \lem{Sgib1981} and $\bbE[T^2] < \infty$ yield, for $t \to \infty$,
\begin{equation}
\label{eq:R2 1}
  \bbE[R^2(t)] \sim \lambda h(t) \sim
\begin{cases}
\lambda t z_2(t) = o(t), &  \mbox{ if } \dd{E}[T^{3}] = \infty, \\
\frac 13 \lambda \dd{E}[T^{3}],  &   \mbox{ if } \dd{E}[T^{3}] < \infty.
\end{cases}
\end{equation}
Thus, $\bbE[R^2(t)] = o(t)$. On the other hand, by the Cauchy--Schwarz
inequality,
\begin{equation}
\label{eq:RM 1}
  \big|\bbE\big[R(t) M(t)\big]\big| \le \sqrt{\bbE[R^2(t)] \,\bbE[M^2(t)]}
\sim \lambda^2 \sqrt{t \sigma_T^2 h(t)} .
\end{equation}
Hence, because $\bbE[T^2] < \infty$, the relations \eq{variance 1} and
$h(t) = o(t)$ yield the known result (e.g.\ \cite[\S2]{Dale2016})
\begin{equation}
\label{eq:variance 0a}
  \var N(t) = \lambda^3 \sigma_T^2 t + o(t), \qquad t \to \infty.
\end{equation}
We can refine \eq{variance 0a} by evaluating $\bbE\big[R(t) M(t)\big]$,
namely, using \eq{R2 1} and \eq{RM 1}, we have the next result.
\begin{proposition}
\label{pro:R2 2}
Let $N$ be a renewal process for which {\rm(A.1)--(A.3)} hold. 
If\/ $\bbE[T^2] < \infty$ and $\bbE[T^3] = \infty$, then with $z_2$ defined by
$\eq{z 2}$, the relation $\eq{variance 0a}$ is tightened to
\begin{align}
\label{eq:variance 2}
    \var N(t) - \lambda^3 \sigma_T^2 t = O\big(t\sqrt{z_2(t)}\big).
\end{align}
\end{proposition}

Now consider the case that $\bbE[T^3] < \infty$. Then from \eq{R2 1}, 
\begin{equation}
\label{eq:R2 asym 2}
  \lim_{t\to\infty} \var R(t)
  = \lim_{t\to\infty} h(t) - \lim_{t\to\infty}(\bbE[R(t)])^2 
  = \tfrac 13 \lambda \bbE[T^3] - \tfrac 14 \lambda^2 \big(\bbE[T^2]\big)^2.
\end{equation}
We want to find the asymptotic behaviour of $\bbE\big[R(t) M(t)\big]$,
but to do so we need an extra condition,
{\parindent 1 pt
\begin{equation}
\label{eq:directly 1}
\mbox{(4a)} \qquad\qquad\quad \bbE[R(t)] - C \mbox{ is directly Riemann integrable on } [0,\infty),
\qquad\qquad\quad
\end{equation}

} 
\noindent
where $C = \tfrac12 \lambda \bbE[T^2] = \lim_{t\to\infty} \bbE[R(t)]$.
Then the following holds (the proof is
given in Appendix A.2).
\begin{lemma}
\label{lem:RM 1}
Assume that conditions {\rm(A.1)--(A.3)} and (4a) hold.
Then when $\bbE[T^3] < \infty$,
\begin{equation}
\label{eq:RM 3}
  \lim_{t\to\infty} \bbE\big[R(t) M(t)\big] = \tfrac 12 (\lambda \bbE[T^2] 
  - \lambda^2 \bbE[T^3]) + \tfrac 12 \lambda^3 \sigma_T^2 \bbE[T^2].
\end{equation}
\end{lemma}

From this lemma and \eq{R2 asym 2}, equation \eq{variance 1} now yields
\eq{variance 4}.
\begin{proposition}
\label{lem:variance 4}
Assume that conditions {\rm(A.1)--(A.3)} and (4a) hold.  
Then if\/ $\bbE[T^3] < \infty$,
\begin{equation}
\label{eq:variance 4}
\var N(t) - \lambda^3 \sigma_T^2 t = - \tfrac23 \lambda^3 \bbE[T^3] + \tfrac 54
\lambda^4 \big(\bbE[T^2]\big)^2 - \tfrac 12 \lambda^2 \bbE[T^2] + o(1),
\quad (t\to\infty).
\end{equation}
\end{proposition}

This result was obtained first by \citet{Smit1959} under the condition that the
distribution $F$ of $T$ is spread out.

\citet{DaleMoha1978} proposed two
conditions $A_{\epsilon}$ and $B_{\rho}$ as below.
\begin{mylist}{5}
\item [\textit{Condition} $A_\epsilon\,.$] For some $\epsilon \ge 0$.
\begin{equation}
\label{eq:R-C limit 2}
  \bbE[R(t)] - C = o(t^{-1-\epsilon}), \qquad t \to \infty,
\end{equation}
\item [\textit{Condition} $B_\rho\,.$] $F$ is strongly nonlattice, that is,
\begin{equation*}
  \liminf_{|\theta| \to \infty} |1 - \varphi_F(\theta)| > 0,
\end{equation*}
and $0 < \bbE[T^\rho] < \infty$ for some $\rho \ge 2$, where $\varphi_F$ is
the characteristic function of $F$, namely,
$\varphi_F(\theta) = \bbE[\re^{\imath\theta T}]$ for $\theta\in\dd{R}$ with
$\imath = \sqrt{-1}$. 
\end{mylist}
Now the spread out condition implies that $F$ is strongly nonlattice
(see e.g.\ \cite[Chapter VII, Proposition 1.6]{Asmu2003}), so when
$\bbE[T^3]<\infty$, \citet{DaleMoha1978}'s Condition $B_\rho$ is weaker
than Smith's assumption \cite{Smit1959}. 

It is easy to see that (4a) is satisfied if either Condition
$A_\epsilon$ holds for $\epsilon>0$ or Condition $B_\rho$ holds
(see e.g.\ \cite[(2.5a)]{DaleMoha1978}). However a function $f(t)=o(t^{-1})$
need not be directly Riemann integrable. Hence, (4a) may be stronger
than Condition $A_0$, though it is unclear whether this case can occur.
On the other hand, \citet[Corollary 1]{Dale2016} shows that 
\begin{equation}
\label{eq:R-C limit 1}
  \lim_{t\to\infty} \int_0^t \big(\bbE[R(u)] - C\big) \, \rd u \hspace{1ex}
\mbox{ exists and is finite},
\end{equation}
if and only if $\bbE[T^3] <\infty$. 
Thus, we may conjecture that $\bbE[T^3] < \infty$ implies (4a),
but this is a hard problem because $\bbE[R(t)] - C$ may oscillate wildly
around the origin as $t\to\infty$.  In other words,
we do not know how to compare (4a) with Condition $A_0$.

Thus, the semi-martingale decomposition \eq{decomposition 3a} can be
used to study the asymptotic behaviour of a higher moment of $N(t)$, but it
appears to require an extra condition such as (4a).


\section{Extension to a general counting process}
\label{sect:extension}
\setnewcounter

The present martingale approach is easily adapted to a general counting process
as long as $D_Y(t)$ of \eq{DY 1} vanishes. Here we consider such an
extension, assuming (A.1), (A.2) and (A.4). Recall that $\dd{F}^X \preceq
\dd{F}$ stands for $\sr{F}^X_{t} \subset \sr{F}_t$ for all $t \ge 0$, where
$X(t) = \big(N(t),R(t)\big)$. Our basic idea is to use a condition similar to
(\sect{martingale}c) (see condition (\sect{extension}a) later).

First we introduce a random function to replace $\lambda$ in
\eq{decomposition 3a} for $v > 0$. 
Let $T^{(v)}_n = T_n \wedge v$, and define $\widetilde{\lambda}^{(v)}(t)$ by
\begin{equation}
\label{eq:general lambda 1}
  \widetilde{\lambda}^{(v)}(t) = \frac 1{\bbE[T^{(v)}_{N(t)}
  \mid \sr{F}_{t_{N(t)-1}-}]}, \qquad t \ge 0,
\end{equation}
equivalently,
\begin{equation*}
  \widetilde{\lambda}^{(v)}(t)
= \frac 1{\bbE[T^{(v)}_n\mid\sr{F}_{t_{n-1}-}]}, \qquad t \in [t_{n-1}, t_n),
\quad n=1,2,\ldots\,.
\end{equation*}
By the orderliness of $N(\cdot)$, $\bbE[T_n^{(v)}\mid\sr{F}_{t_{n-1}-}]$ is finite
and positive, so $\widetilde{\lambda}^{(v)}(t)$ is finite and bounded below by
$1/v$, and is therefore well defined.

\begin{lemma}
\label{lem:decomposition 3}
Let $\dd{F}$ be a filtration such that $\dd{F}^{X} \preceq \dd{F}$, and assume {\rm(A.1)}, {\rm(A.2)} and {\rm(A.4)}.  Then the counting process $N(t)$ can be
decomposed for each $v>0$ via $R^{(v)}(t)= R(t)\wedge v$ as
\begin{equation}
\label{eq:representation 3}
  N(t) = \int_0^t \widetilde{\lambda}^{(v)}(s)\, 1\big(R(s)\le v\big)\,\rd s
+ \widetilde{\lambda}^{(v)}(t) R^{(v)}(t) + M^{(v)}(t),
\end{equation}
where 
\begin{equation}
\label{eq:martingale 3}
  M^{(v)}(t) = \sum_{n=1}^{N(t)} \big(1 - \widetilde{\lambda}^{(v)}(t_{n-1}) T^{(v)}_n\big), \qquad t \ge 0,
\end{equation}
is an $\dd{F}$-martingale.
\end{lemma}

\begin{remark} \label{rem:4.1}
{\rm The left-hand side of \eq{representation 3} does not depend on $v$, and 
therefore the right-hand side is also independent of $v$ (see \eq{truncated 3}
and arguments below it).  We note that \lem{decomposition 3} holds
for $v=\infty$, and can be regarded as an extension of Theorem 2.1. }
\end{remark}

\begin{proof}
Apply \lem{martingale 1} with $Y(t) = N(t) -\wtillambda^{(v)}(t) R^{(v)}(t)$.
Note that
\begin{equation*}
  Y'(t) = \wtillambda^{(v)}(t) \,1\big(R(t) \le v\big), \qquad t_{n-1} <t<t_n,
\end{equation*}
because $\wtillambda^{(v)}(t)$ is piecewise constant and $R'(t)= -1$
for $t \in (t_{n-1}, t_n)$.  The facts that $R^{(v)}(t_n-)=0$,
$R^{(v)}(t_n) = T^{(v)}_{n+1}$ and $\widetilde{\lambda}^{(v)}(t_n)$ is
$\sr{F}_{t_n-}$-measurable, imply that
\begin{align}
\label{eq:DY 3}
 D_Y(t)  & =\sum_{n=0}^{N(t)-1} \bbE\big[1 - \widetilde{\lambda}^{(v)}(t_n)
T^{(v)}_{n+1} \mid\sr{F}_{t_n-}\big]  \nonumber\\
 &  = \sum_{n=0}^{N(t)-1} \big(1 - \widetilde{\lambda}^{(v)}(t_n)
\bbE[T^{(v)}_{n+1}\mid\sr{F}_{t_n-}]\big) = 0.
\end{align}
On the other hand,
\begin{equation*}
  M_Y(t) = 1 - \widetilde{\lambda}^{(v)}(t_0) T^{(v)}_1
+ \sum_{n=1}^{N(t)-1} \big(1 - \widetilde{\lambda}^{(v)}(t_n) T^{(v)}_{n+1}\big)
= \sum_{n=1}^{N(t)} \big(1 - \widetilde{\lambda}^{(v)}(t_{n-1}) T^{(v)}_n\big).
\end{equation*}
Finally, 
\begin{equation*}
  \dd{E}\big[|\Delta Y(t_{n})|\bigm|\sr{F}_{t_{n}-}\big] \le
\dd{E}\big[1+\widetilde{\lambda}^{(v)}(t_n) T^{(v)}_{n+1}\bigm|\sr{F}_{t_{n}-}\big] = 2.
\end{equation*}
Hence, (B.4) is satisfied, and therefore $M^{(v)}(\cdot) \equiv M_{Y}(\cdot)$ is an $\dd{F}$-martingale by \lem{martingale 1}, completing the proof of \lem{decomposition 3}.
\end{proof}

Thus, we have derived the semimartingale representation \eq{representation 3}
for $N(\cdot)$ under the assumptions (A.1), (A.2) and (A.4). Using this
representation, we extend Blackwell's renewal theorem to a general counting
process.  To do this, we focus attention on Condition (\sect{martingale}c) of
\lem{Renewal 1}, of which the following can be viewed as its extended
version.
\begin{itemize}
\item [(\sect{extension}a)] There exists $v>0$ such that as $t\to\infty$,
$\bbE\big[\wtillambda^{(v)}(t)\,1\big(R(t)\le v\big)\big]$ and
$\bbE\big[\wtillambda^{(v)}(t)\,R^{(v)}(t)\big]$ 
converge to finite positive limits. 
\end{itemize}
Since $\bbE[T^{(v)}_{N(t)}\mid\sr{F}_{t_{N(t)-1}-}]$ is bounded by $v>0$, it
may be easier to check condition (\sect{extension}a) via the weak convergence
of $\bbE[T_{N(t)}^{(v)} \mid \sr{F}_{t_{N(t)-1}-}]$ as $t\to\infty$, but to do
this we need an extra condition of uniform integrability: the following is
sufficient for (\sect{extension}a).
\begin{itemize}
\item [(\sect{extension}b)] There exists $v>0$ such that 
\begin{enumerate}
\item[(\sect{extension}b-i)] $v$ is a continuity point of the limit
distribution of $R^{(v)}(t)$, 
\item[(\sect{extension}b-ii)]
$\big(\bbE[T^{(v)}_{N(t)} \mid \sr{F}_{t_{N(t)-1}-}],\, R^{(v)}(t)\big)$ has a
limiting distribution as $t\to\infty$, and 
\item[(\sect{extension}b-iii)] $\{\wtillambda^{(v)}(t):t\ge 0\}$ is uniformly integrable, i.e.\
 $$\lim_{a\to\infty} \sup_{t\ge 0} \bbE\big[\wtillambda^{(v)}(t)\,1\big(
   \wtillambda^{(v)}(t) > a\big)\big] = 0.$$ 
\end{enumerate}
\end{itemize}

We now present a general conclusion from (\sect{extension}a).
\begin{theorem}
\label{thr:Blackwell general 1}
Under the assumptions of \lem{decomposition 3}, if Condition
(\sect{extension}a) holds, then there exists $\lambda > 0$ such that both
\begin{equation}
\label{eq:lambda 1}
  \lim_{t \to \infty} \bbE[N(t)]\big/t = \lambda
\end{equation}
and, for $0<h<\infty$,
\begin{equation}
\label{eq:Blackwell general 1}
   \lim_{t \to \infty} \big(\bbE\big[N(t+h)\big]
   - \bbE\big[N(t)\big]\big) = \lambda h.
\end{equation}
\end{theorem}
\begin{proof}
Let $v > 0$ be such that the expectations in Condition (\sect{extension}a)
converge; then by (\sect{extension}a) there exists $\lambda^{(v)}>0$ such that
\begin{equation}
\label{eq:lambda 2}
  \lim_{t\to\infty} \bbE\big[\wtillambda^{(v)}(t)\,1\big(R(t) \le v\big)\big]
   =\lambda^{(v)}.
\end{equation}
Apply \lem{decomposition 3}. Taking the expectation of \eq{representation 3}
yields
\begin{equation}
\label{eq:decomposition 4}
 \bbE[N(t)] 
  = \int_0^t \bbE\big[\wtillambda^{(v)}(s) \,1\big(R(s)\le v\big)\big] \,\rd s 
  + \bbE\big[\wtillambda^{(v)}(t) \,R^{(v)}(t)\big].
\end{equation}
Divide both sides of this equation by $t$; letting $t \to \infty$ yields
$\lim_{t \to \infty} \bbE[N(t)]\big/t = \lambda^{(v)}$.
Now the left-hand side of this relation is independent of $v$, so
$\lambda^{(v)}$ must also be independent of $v$: set $\lambda = \lambda^{(v)}$.
We thus have \eq{lambda 1}, while \eq{Blackwell general 1} follows from
\eq{decomposition 4} and (\sect{extension}a). 
\end{proof}

In applying Theorem \ref{thr:Blackwell general 1} it is important to check
Condition (\sect{extension}a) or (\sect{extension}b).  Obviously, Conditions
(\sect{extension}b) are satisfied by a non-arithmetic renewal process (see
Assumptions (A.1)--(A.3)), for which $T^{(v)}_n$ is identically distributed
and independent of $\sr{F}_{t_{n-1}-}\,$. We sketch two scenarios in which
the two conditions are relaxed.

\subsection{Modulated inter-arrival times}
\label{sect:modulated}

Let $J(\cdot) \equiv \{J(t); t \ge 0\}$ be a piecewise constant process on the
state space $S$ which is a Polish space.  Let $t_0 = 0$, and for
$n=1,2,\ldots\,$
let $t_n$ be the $n^{\rm th}$ discontinuous instant of $J(t)$; these instants
generate the counting process $N(\cdot)$. As usual, let $T_n = t_n-t_{n-1}$ and
for $t\in [t_{n-1},t_n)$ set $R(t) = t-t_{n-1}$ and $J(t) = J(t_{n-1})$.
Define a joint process $U(\cdot)$ by   
\begin{equation*}
  U(t) = \big(J(t), N(t), R(t)\big), \qquad t \ge 0.
\end{equation*}
Let $\sr{F}^{U}_t = \sigma(\{U(s); s\le t\})$ and let
$\dd{F}^U = \{\sr{F}^U_t; t \ge 0\}$; this is a filtration for
$U(\cdot)$. Let $\dd{F} = \dd{F}^U$, then obviously $\dd{F}^X \preceq \dd{F}$
since $X(t) = \big(N(t),R(t)\big)$. Assume the following conditions:
\begin{itemize}
\item [(M1)] $T_n$ is independent of $\sr{F}_{t_{n-1}-}\,$;
\item [(M2)] the distribution of $T_n$ is non-arithmetic and determined by
$J(t_{n-1}) \in S$.
\end{itemize}

We refer to a process satisfying (M1) and (M2) as a {\it modulated renewal
process}.  A Markov modulated renewal process is the special case in which
$\{J(t_n): n=0,1,\ldots\}$ is a Markov chain.  Let $T^{(v)}(x)$ be the
conditional expectation of $T^{(v)}_{n} \equiv v \wedge T_n$ given
$J(t_{n-1})=x$, that is,  $T^{(v)}(x) = \bbE[T^{(v)}_n\mid J(t_{n-1})=x]$.

\begin{corollary}
\label{cor:modulated RP}
For a modulated renewal process as defined above, if 
{\rm(i)} $S$ is countable, {\rm(ii)} $\inf_{x\in S}\bbE[T(x)] >0$, where
$T(x) = \bbE[T_n\mid J(t_{n-1})=x]$, and {\rm(iii)} $\big(J(t),R(t)\big)$
has a limit distribution as $t \to \infty$, then both \eq{lambda 1} and
Blackwell's formula \eq{Blackwell general 1} hold with $\lambda$ defined by
\begin{equation}
\label{eq:lambda 4}
  \lambda
 = \bbE\bigg[\frac {1} {T^{(v)}(\widetilde{J})} 1(\widetilde{R} \le v) \bigg]
= \bbE\bigg[\frac {1} {T(\widetilde{J})}\bigg],
\end{equation}
where $(\widetilde{J},\widetilde{R})$ is a r.v.\ with the limit distribution
of $\big(J(t), R(t)\big)$, and $v$ is any continuity point of the
distribution of $\widetilde{R}$. 
\end{corollary}
\begin{proof}
From condition (iii), 
for a continuity point $v$ of the distribution of $\widetilde{R}$, and for a
bounded function $f: S \mapsto \bbR$, 
 \begin{align*}
 \lim_{t\to\infty} \bbE\big[ f\big(J(t)\big)\,1\big(R(t)\le v\big)\big]
  &= \bbE\big[f(\widetilde{J})\,1(\widetilde{R}\le v)\big],\\
  \lim_{t\to\infty} \bbE\big[f\big(J(t)\big)\,R^{(v)}(t)\big]
  &= \bbE\big[f(\widetilde{J}) (\widetilde{R}\wedge v )\big].
\end{align*}
By assumptions (i) and (ii) of the corollary, $f(x) := 1/\bbE[T^{(v)}(x)]$ is
continuous, bounded and positive, where $x$ is discrete, so we take a discrete
topology.  Thus, Condition (\sect{extension}a) is satisfied, and therefore
\eq{lambda 1} and \eq{Blackwell general 1} are obtained by
\thr{Blackwell general 1}. Here, \eq{lambda 4} is immediate from \eq{lambda 2}
in the proof of \thr{Blackwell general 1}.
\end{proof}

\begin{remark} \label{rem:4.3}
{\rm Under the conditions of \cor{modulated RP}, $J(\cdot)$ is piecewise
continuous but no transition structure like that of a Markov chain is assumed:
the restrictive conditions (i) and (ii) may be inconsistent with Markovianity.
If $S$ is a finite set then (i) and (ii)
automatically hold, and these may constitute circumstances when the present
framework is useful.  However, for a Markov modulated renewal process, 
Blackwell's formula \eq{Blackwell general 1} can be obtained under a certain
recurrence condition of $J(t)$ without conditions (i) and (ii) (see e.g.\
\cite{Alsm1997}).  In such a case the present approach would not be suitable. }
\end{remark}

\subsection{Stationary inter-arrival times}
\label{sect:stationary IA}

Consider now the scenario in which $\{T_n; n \in \bbZ_+\}$ is a stationary
sequence of positive reals with finite means, where $\bbZ_+$ is the set of
all non-negative integers. This sequence can be extended to a stationary
sequence that starts at time $-\infty$, and is well described by the Palm
distribution $\bbP$ on a measurable space $(\Omega, \sr{F})$.
[We digress to note that in the point process literature, the Palm distribution
is often notated as $\bbP_0$, and if need be, the distribution of a
stationary point process (i.e.\ for which
the distributions of counts on sets $A_n$ are
the same as for the translated sets $A_n+t$) are notated $\bbP$. To be
consistent with Sections 1--3 of this paper we retain the notation $\bbP$ for
Palm distributions, and write $\barbbP$ for (count) stationary
distributions as at \eq{cycle 1} below.]

We introduce the standard formulation to describe $\{T_n\}$ by a point process
under $(\Omega,\sr{F},\dd{P})$ (see e.g.\ \cite{BaccBrem2003}). Let
$\lambda = 1/\bbE[T_0]$, and let $\{t_n\}$ be a two-sided random sequence
such that $t_0 = 0$ and
\begin{equation*}
  t_n = \begin{cases}
T_1+ \cdots + T_n ,  & n > 0,  \\
-(T_{-1}+ \cdots + T_n) ,  & n < 0.
\end{cases}
\end{equation*}
Define a point process $N(\cdot)$ on $\bbR$ via sets $B \in \sr{B}(\dd{R})$ by
\begin{equation*}
  N(B) = \sum_{n=-\infty}^\infty 1(t_n \in B).
\end{equation*}
Similar to \eq{Rdef}, we define $R(t)$ as
\begin{equation*}
  R(t) = \begin{cases}
 \sum_{\ell=1}^{N([0,t])} T_{\ell} - t  & \text{for }t\ge 0, \\
\noalign{\smallskip}
 \sum_{\ell=1}^{N((t,0))} (-T_{-\ell}) + t  &\text{for }t < 0.
\end{cases}
\end{equation*}

We can then construct a shift operator group $\{\theta_{t}; t\in\bbR\}$ on
$\Omega$ such that
\begin{itemize}
\item [(S1)] $\theta_t \circ A = \{\omega\in\Omega: \theta_t^{-1}(\omega)\in A\}$;
\item [(S2)] the point process $N$ is consistent with $\theta_t$, that is,
$\theta_t \circ N(B) = N(B+t)$ for bounded $B \in \sr{B}(\dd{R})$ and
$B+t = \{x+t \in \dd{R}: x \in B\}$; and
\item [(S3)] for $n \in \bbZ$, $\bbP(\theta_{t_n} \circ A) = \bbP(A)$ for
$A \in \sr{F}$, where $\bbZ$ is the set of all integers.
\end{itemize}
 
Next define a probability measure $\barbbP$ on $(\Omega,\sr{F})$ by
\begin{equation}
\label{eq:cycle 1}
 \barbbP(A) = \lambda \bbE\bigg[\int^{T_1}_0 \theta_t \circ 1_A \ \rd t\bigg],
\qquad A \in \sr{F}.
\end{equation}
It is well known (see e.g.\ \cite{BaccBrem2003})
that $N(\cdot)$ is a stationary point process under
$\barbbP$, and $\ol{\bbE}[N(1)] = \lambda$. Furthermore, we recover $\bbP$
from $\barbbP$ by the so-called inversion formula: for each $\epsilon > 0$
\begin{equation}
\label{eq:inversion 1}
  \bbP(A) = \frac 1{\lambda \epsilon} \barbbP\bigg[\int_0^\epsilon
\theta_{-t} \circ 1_A\ N(\rd t) \bigg], \qquad A \in \sr{F}.
\end{equation}

We can now formulate Blackwell's renewal theorem for the stationary sequence.

\begin{corollary}
\label{cor:Palm 1}
Under assumptions {\rm(A.1)--(A.2)}, if {\rm(i)} $\{T_n; n \in \bbZ\}$ is a
stationary and ergodic sequence under the Palm distribution $\bbP$,
{\rm(ii)} $\{\wtillambda^{(v)}(t): t\ge 0\}$ is uniformly integrable under
$\bbP$, and {\rm(iii)} the mixing condition
\begin{equation}
\label{eq:mixing 1}
  \lim_{t\to\infty} \ol{\bbP}(\theta_{-t} \circ A, t_{1} \le u) 
= \barbbP(A) \,\barbbP(t_1 \le u), \qquad A \in \sr{F}, u \ge 0,
\end{equation}
holds, then Blackwell's formula \eq{Blackwell general 1} holds together with
\eq{lambda 1}, and
\begin{equation} \label{eq:lambda 3}
    \lambda = \frac{1}{\bbE[T_1]}
  = \bbE\bigg[\frac{1\big(R(0) \le v\big)}{\bbE[T_1^{(v)}\mid \sr{F}_{0-}]}
\bigg], \qquad v\ge 0,
\end{equation}
where the filtration $\dd{F} \equiv \{\sr{F}_{t}; t \in \dd{R}\}$ is given by $\sr{F}_{t} = \sigma\left(\{R(u); u \le t\} \cup \bigcup_{n=-\infty}^{\infty} \{t_{n} \le t\}\right)$. 
\end{corollary}

\begin{remark} 
\label{rem:Palm 1}
{\rm The mixing condition \eq{mixing 1} is used
in \cite[Theorem 3.2]{Miya1977}.}
\end{remark}

\begin{proof}
Let $\eta_n(\omega) = \theta_{t_n(\omega)}(\omega)$ be the shift operator on
the sample space $\Omega$; then
\begin{align*}
  \eta_1 \circ \big(\widetilde{\lambda}^{(v)}(t), R^{(v)}(t)\big)
& = \bigg(\frac 1{\eta_1 \circ \bbE[T^{(v)}_n\mid\sr{F}_{t_{n-1}-}]},
\,(\eta_1 \circ T_n - (t- \eta_1 \circ t_{n-1})) \wedge v \bigg)\\
  & = \bigg(\frac 1{\bbE[T^{(v)}_n\mid\sr{F}_{t_n-}]},
 \, R^{(v)}(t_{n}) \bigg), \qquad t_n \le t < t_{n+1}.
\end{align*}
Hence for $n\in\bbZ$,
$ \big\{\big(\widetilde{\lambda}^{(v)}(t), R^{(v)}(t)\big);
t_{n-1}\le t < t_n\big\}$
is a stationary sequence under $\bbP$, and 
\begin{equation*}
  \big(\widetilde{\lambda}^{(v)}(t),\, R^{(v)}(t)\big)
= \theta_t \circ \big(\widetilde{\lambda}^{(v)}(t) \circ
   \theta_{-t},\, R^{(v)}(t) \circ \theta_{-t})\big)
= \theta_t \circ \big(\widetilde{\lambda}^{(v)}(0), R^{(v)}(0)\big)
\end{equation*}
is a stationary process under $\ol{\bbP}$. Let $f(x,y)$ be a non-negative
bounded continuous function on $\bbR_+^2$; then by \eq{mixing 1}, for
$\epsilon > 0$,
\begin{equation*}
  \lim_{t\to\infty} \ol{\bbE}\big[f\big(\wtillambda^{(v)}(t), \,R^{(v)}(t)\big)\, 1(t_1 \le \epsilon)\big]
= \ol{\bbE}\big[f\big(\wtillambda^{(v)}(0), \,R^{(v)}(0)\big)\big]\,
\ol{\dd{P}}(t_1 \le \epsilon).
\end{equation*}
On the other hand, by \eq{inversion 1},
\begin{align*}
  &\bigg|\bbE\big[ f\big(\wtillambda^{(v)}(t), \,R^{(v)}(t)\big)\big]
  - \frac {\ol{\bbE}\big[f\big(\wtillambda^{(v)}(t+t_1), \,R^{(v)}(t+t_1)\big) 
\,1(t_1 \le \epsilon) \big]} {\lambda \epsilon}
 \bigg|\\
  & \quad  \le \frac {\|f\|}{\lambda \epsilon}
  \ol{\bbE}\big[N(\epsilon)\,1\big(N(\epsilon) \ge 2\big)\big],
\end{align*}
where $\|f\| = \sup_{(x,y) \in \dd{R}_+^2} |f(x,y)|$. This and \eq{mixing 1}
imply that
\begin{align*}
 &\quad\; \limsup_{t\to\infty} \bbE\big[ f\big(\wtillambda^{(v)}(t),
     \,R^{(v)}(t)\big)\big] \\
 &\le\limsup_{t\to\infty} \frac 1{\lambda \epsilon} \ol{\bbE}\Big[
\sup_{s \in [0,\epsilon)} f\big(\wtillambda^{(v)}(t+s), \,R^{(v)}(t+s)\big)
\, 1(t_1 \le \epsilon) \Big]   
  + \frac {\|f\|}{\lambda \epsilon}
   \ol{\bbE}\big[N(\epsilon)\,1\big(N(\epsilon) \ge 2\big)\big]\\
  &= \frac 1{\lambda \epsilon}
\ol{\bbE}\Big[\sup_{s \in [0,\epsilon)} f\big(\wtillambda^{(v)}(s), 
\,R^{(v)}(s)\big)\Big]\, \ol{\dd{P}}(t_1 \le \epsilon) 
  + \frac {\|f\|}{\lambda \epsilon}
\ol{\bbE}\big[N(\epsilon)\,1(N(\epsilon) \ge 2)\big].
\end{align*}
Now   
\begin{equation*}
  \lim_{\epsilon\downarrow 0} \frac {\ol{\dd{P}}(t_1 \le \epsilon)}{\lambda\epsilon}
= 1\qquad\hbox{and}\qquad
  \lim_{\epsilon \downarrow 0}
\frac{\ol{\bbE}\big[N(\epsilon)\,1(N(\epsilon) \ge 2)\big]}{\lambda\epsilon}=0,
\end{equation*}
implying that
\begin{equation*}
 \limsup_{t\to\infty} \bbE\big[f\big(\wtillambda^{(v)}(t),
 \,R^{(v)}(t)\big)\big] 
\le \lim_{\epsilon\downarrow 0} \ol{\bbE}\Big[ \sup_{s \in [0,\epsilon)}
 f\big(\wtillambda^{(v)}(s), \,R^{(v)}(s)\big)\Big] 
   = \ol{\bbE}\big[f\big(\wtillambda^{(v)}(0), \,R^{(v)}(0)\big)\big],
\end{equation*}
by the right-continuity of $f\big(\wtillambda^{(v)}(t), \,R^{(v)}(t)\big)$.
Similarly, we have
\begin{equation*}
  \liminf_{t\to\infty} \bbE\big[f\big(\wtillambda^{(v)}(t),
 \, R^{(v)}(t)\big)\big] 
\ge \ol{\bbE}\big[ f\big(\wtillambda^{(v)}(0), \,R^{(v)}(0)\big)\big].
\end{equation*}
Hence, the distribution of $\big(\wtillambda^{(v)}(t), \,R^{(v)}(t)\big)$ under
$\dd{P}$ converges weakly to $\big(\wtillambda^{(v)}(0), \,R^{(v)}(0)\big)$
under $\ol{\dd{P}}$. Then we can choose their almost surely convergent version
by the Skorohod representation theorem. This and the uniform integrability
assumption (ii) conclude condition (\sect{extension}a) because $R^{(v)}(0)$ is
bounded by $v$.  Equation \eq{lambda 3}
is an immediate consequence of \eq{cycle 1}, completing the proof.
\end{proof}

\section{Concluding remarks}
\label{sect:concluding}
\setnewcounter

In this paper, we  have used a certain martingale to give a new approach to
Blackwell's renewal theorem and its extensions for general counting processes.
One may envisage applying this approach to other problems.


For example, consider a diffusion approximation of the renewal process $N$ of
Section 2 for which $\bbE[T^2]<\infty$. 
Scale $N(t)-\lambda t$ as $\widetilde{N}_n(t)
:= n^{-1/2} \big(N(nt)-\lambda n t\big)$;
this is called diffusion scaling. It is well known that $\widetilde{N}_n(\cdot)$
converges weakly to the Brownian motion $B(t)$ with 
$\var B(t) = \lambda^3\sigma^2_T t$ in an appropriate function space with the
Skorokhod topology.  This is usually proved by the central limit theorem
and a time change (see e.g.\
\cite[Theorem 5.11]{ChenYao2001} and \cite[Corollary 7.3.1]{Whit2002}).
To derive this result in the framework of this paper, let
$\widetilde{R}_n(t) = \lambda R(nt)/\sqrt{n}$ and 
$\widetilde{M}_n(t) = M(nt)/\sqrt{n}$; then by Theorem 2.1, 
 \begin{equation} \label{eq:MMM 1}
 \widetilde{N}_n(t) = \lambda \widetilde{R}_n(t) + \widetilde{M}_n(t).
\end{equation}
Observe that as $n\to\infty$, $\widetilde{R}_n(t)\to 0$ in probability because
 \begin{equation*}
 \limsup_{n\to\infty} \bbE[\widetilde{R}_{n}(t)]
  \le \limsup_{n\to\infty} \lambda \bbE[T^2]/\sqrt{n} = 0.
\end{equation*}
Further, Lemma 2.3 implies that almost surely,
\begin{equation} \label{eq:MMM 2}
 \lim_{n\to\infty} \langle \widetilde{M}_n(t) \rangle
  = \lim_{n\to\infty} \frac{N(nt)}{n}\,\lambda^2 \sigma^2_T 
  = \lambda^3 \sigma^2_T t.
\end{equation}
Hence \eq{MMM 1} would imply that $\widetilde{N}_n(\cdot)$ converges weakly to 
the martingale with deterministic quadratic variation $\sigma^2_T t$, and
this is just the Brownian motion $B(t)$ with variance $\sigma^2_T t$ if the limiting process of $ \widetilde{N}_n(\cdot)$ is continuous in time.
To convert this argument into a formal proof, we should need to verify a
technical condition, called $C$-tightness (see e.g.\
\cite[Theorem 2.1]{Whit2007a}); even without this, the above argument
elucidates the mechanism of the diffusion approximation.

Thus, while the present approach is useful for studying counting processes, it
may not be the case for studying stochastic models in applications.
For example, counting processes appear as input data for stochastic models such
as queueing and risk processes. In these applications, the semimartingale representation for a counting process may not be convenient because such stochastic
processes are functionals of counting processes. In this situation, the general
formulation in \sectn{deriving} would be useful if we can find an appropriate
process $Y(t)$, which may not be a counting process but includes it as one of
components. The second author recently studied this type of application in
\cite{Miya2017,Miya2017a} for diffusion approximations and tail asymptotics of
the stationary distributions for queues and their networks.
This may be a direction for future study.

\appendix

\section*{Appendix}
\label{app:proofs}
\setnewcounter

\setcounter{section}{1}

\subsection{Proof of \eq{M2 2}}
\label{app:M2 2}

Apply \lem{martingale 1} with $Y(t) = M^2(t)$ for which $Y'(t)=0$, and therefore
\begin{align*}
  \langle M\rangle(t) & = D_Y(t) = \sum_{n=0}^{N(t)-1} \big(\bbE[M^2(t_n) \mid
\sr{F}_{t_n-}] - M^2(t_n-)\big)\\
  &= \sum_{n=0}^{N(t)-1}\big(\bbE\big[\big(M(t_n-) + (1-\lambda T_{n+1})\big)^2 \mid
\sr{F}_{t_n-}\big] - M^2(t_n-)\big)\\ 
  &= \sum_{n=0}^{N(t)-1} \bbE\big[2M(t_n-) (1-\lambda T_{n+1}) +(1-\lambda T_{n+1})^2 
\mid \sr{F}_{t_n-}\big]\\
  &= \sum_{n=0}^{N(t)-1} \dd{E}\left[(1-\lambda T_{n+1})^{2}\right]\\
  &= N(t) \dd{E}\left[(1-\lambda T)^{2}\right] = \lambda^{2} \sigma_{T}^{2} N(t),
\end{align*}
since $\bbE[1-\lambda T_{n+1}\mid\sr{F}_{t_n-}] = 0$. Thus, we have \eq{M2 2}.

\subsection{Proof of \lem{RM 1}}
\label{app:RM 1}

The main idea of this proof is to apply the key renewal theorem. For this,
recall that $t_{N(t)-1} \le t < t_{N(t)}$, and rewrite $R(t) M(t)$ as
\begin{equation}
\label{eq:RM 2}
  R(t)M(t) = (t_{N(t)} - t) \sum_{\ell=1}^{N(t)} (1 - \lambda T_\ell) 
 = Z_1(t) + Z_2(t),
\end{equation}
where
\begin{align*}
   Z_1(t) &= (T_{N(t)} + t_{N(t)-1} - t) (1 - \lambda T_{N(t)}),\\
   Z_2(t) &= (T_{N(t)} + t_{N(t)-1} - t)
\sum_{\ell=1}^{N(t)-1} (1 - \lambda T_{\ell}).
\end{align*}

We consider $\bbE[Z_1(t)]$ and $\bbE[Z_2(t)]$ separately. Let
\begin{equation*}
  z_3(t) = \bbE[(T-t)(1-\lambda T)\,1(T>t)], \qquad t \ge 0.
\end{equation*}
Then, much as for \eq{gre R 1}, the independence of $t_{n-1}$ and $T_n$ and the
key renewal theorem (see e.g.\ \cite[Example 2.6]{Asmu2003}) yield
\begin{align}
\label{eq:Key 1}
 \lim_{t\to\infty} \bbE\big[Z_1(t)\big]
   & = \lim_{t\to\infty} \bbE\bigg[\sum_{n=1}^\infty \big(T_n -(t-t_{n-1})\big)
(1 - \lambda T_n) \,1(0 \le t-t_{n-1} < T_n)\bigg] \nonumber\\
 & = \lim_{t\to\infty} \bbE\bigg[\sum_{n=1}^\infty z_3(t-t_{n-1})
  \,1(t_{n-1} \le t)\bigg] \nonumber\\
 & = \lambda \int_0^\infty z_3(u) \, \rd u 
 = \lambda \bbE\bigg[\int_0^T (T-u)(1-\lambda T) \, \rd u \bigg] \nonumber\\
 & = \tfrac 12 \lambda \big(\bbE[T^2] - \lambda \bbE[T^3]\big).
\end{align}

In considering $\bbE[Z_2(t)]$, the limiting operations for the key renewal
theorem are nested, so we use the extra condition (4a).
We prove that $\bbE[T^3] < \infty$ and that
\begin{equation}
\label{eq:Key 2}
  \lim_{t\to\infty} \bbE\big[Z_2(t)\big] 
  = \tfrac 12 \lambda^3 \bbE(T^2) \sigma_T^2.
\end{equation}
First rewrite $\bbE[Z_2(t)]$ as
\begin{align}
\label{eq:expand 1}
  \bbE\bigg[(T_{N(t)} + t_{N(t)-1} - t) &\sum_{\ell=1}^{N(t)-1} (1-\lambda T_\ell)\bigg] 
= \bbE\bigg[\sum_{n=1}^\infty (t_n-t) \sum_{\ell=1}^{n-1} (1-\lambda T_\ell)
\,1(t_{n-1} \le t < t_n)\bigg] \nonumber\\
  & = \bbE\bigg[\sum_{\ell=1}^\infty (1 - \lambda T_\ell)
\sum_{n=\ell+1}^\infty (t_n - t) \,1(t_{n-1} \le t < t_n)\bigg]\nonumber\\
  & = \bbE\bigg[\sum_{\ell=1}^{\infty} (1 - \lambda T_{\ell}) V_{\ell}(t)
 \,1(t_{\ell-1} \le t) \bigg],
\end{align}
where
\begin{equation*}
  V_\ell(t) = \bbE\bigg[\sum_{n=\ell+1}^\infty (t_n - t) \,1(t_{n-1} \le t<t_n)
\biggm| \sr{F}_{t_{\ell-1}}\bigg], \qquad t \ge 0, \ell \ge 1.
\end{equation*}

Let $\tN(\cdot)$ be an independent copy of $N(\cdot)$, let
$\widetilde{t}_n$ be the $n^{\rm th}$ counting epoch of the renewal process
$\widetilde{N}(\cdot)$ similar to $N(\cdot)$, and let $\widetilde{T}_n =
\widetilde{t}_n - \widetilde{t}_{n-1}$ for $n \ge 1$, where $\widetilde{t}_0=0$.
Similarly, let $\widetilde{R}(t)$ be the residual time to the next jump of
$\tN(\cdot)$ at time $t$.  For $t\ge0$ define
\begin{align*}
  \widetilde{V}(t\mid x) &= \bbE\Big[\sum_{n=2}^\infty (\widetilde{t}_n - t)
1(\widetilde{t}_{n-1} \le t < \widetilde{t}_n)\,\Big|\, \widetilde{T}_1 
= x \Big]\nonumber\\
  &= \bbE\Big[\sum_{n=2}^\infty \big(\widetilde{t}_n - \widetilde{t}_1 
- (t - \widetilde{t}_1)\big) \,1(\widetilde{t}_{n-1} - \widetilde{t}_1 \le t 
- \widetilde{t}_1 < \widetilde{t}_n - \widetilde{t}_1)\,\Big|\, \widetilde{T}_1 = x \Big].
\end{align*}
Since the last formula is independent of $\widetilde{T}_1$ and represents the
residual time to the next jump at time $t-x$, it equals
$\bbE\big[\widetilde{R}(t-x)\big]$.

For notational convenience in what follows, define a function $r(\cdot)$ by
\begin{equation*}
 r(t) = \begin{cases}
 \dd{E}[\widetilde{R}(t)],  &  t \ge 0, \\
 0 &  t < 0, 
\end{cases}
\end{equation*}
 and as earlier let $C = \frac 12 \lambda \dd{E}[T^2]$. Since for
$n\ge\ell$, $t_n-t_\ell$ is independent of $\sr{F}_{t_{\ell-1}}$, 
\begin{equation*}
  V_\ell(t) = \widetilde{V}(t-t_{\ell-1}\,\mid\, T_\ell) 
  = r\big(t-(t_{\ell-1}+T_\ell)\big), \qquad t \ge 0.
\end{equation*}
Thus, \eq{expand 1} can be rewritten as
\begin{align}
\label{eq:expand 2}
 \bbE\bigg[(T_{N(t)} + t_{N(t)-1} & - t) \sum_{\ell=1}^{N(t)-1}
(1 - \lambda T_{\ell})\bigg]  
   = \bbE\bigg[\sum_{\ell=1}^{\infty} (1 - \lambda T_{\ell})
 r\big(t - (t_{\ell-1}+T_{\ell})\big) \,1(t_{\ell-1} \le t )\bigg].
\end{align}
Denote the right-hand side of  \eq{expand 2} by $W(t)$; decompose it as
  $W(t) = W_1(t) + W_2(t)$,
where
\begin{align*}
 & W_1(t) = \bbE\bigg[\sum_{\ell=1}^\infty (\lambda T_\ell - 1)
\Big(C \,1(T_\ell\le t - t_{\ell-1}) - r\big(t - (t_{\ell-1}+T_\ell)\big)
\Big) 1(t_{\ell-1} \le t )\bigg],\\
 & W_2(t) = C \bbE\bigg[\sum_{\ell=1}^\infty (1 - \lambda T_\ell)
\,1(T_\ell\le t - t_{\ell-1})\bigg].
\end{align*}
Define
\begin{align*}
  w_1(t) = \bbE[(\lambda T - 1)\big(C-r(t-T)\big)
\, 1(T \le t)]\qquad \mbox{and}\qquad 
  w_2(t) = C\, \bbE[(1 - \lambda T)1(T \le t)].
\end{align*}
It is readily checked that, for $i=1,2$, $W_i(\cdot)$ is the solution of the
general renewal equation with the generator $w_i(\cdot)$.
Consider first $w_1(t)$, and introduce
\begin{equation*}
  g(t) = \begin{cases}
   C - r(t), & t \ge 0,  \\
   0, &   t < 0,
\end{cases}
\end{equation*}
so that from the definition of $w_1$,
\begin{align*}
  w_1(t) & = \bbE[\lambda T g(t-T) 1(T \le t)] - \bbE[g(t-T) 1(T \le t)]\\
  & = \lambda\int_0^\infty u g(t-u)\,F(\rd u) - \int_0^\infty g(t-u)\,F(\rd u).
\end{align*}
We show that under condition (4a), the last two integrals are directly Riemann
integrable.  To this end, let $I^\delta_n(u) = (n\delta, (n+1)\delta]$ for
$\delta > 0$ and $u \ge 0$. Then
\begin{equation*}
  \sup_{t\in I^\delta_n(0)} \int_0^\infty u\, g(t-u) \,F(\rd u)
\le \int_0^\infty u\,\sup_{t \in I^\delta_n} g(t-u) \,F(\rd u).
\end{equation*}
Similarly, 
\begin{equation*}
  \int_0^\infty u \inf_{t \in I^\delta_n} g(t-u) \,F(\rd u)
\le \inf_{t \in I^\delta_n(0)} \int_0^\infty u \,g(t-u) \,F(\rd u).
\end{equation*}

Then because $|g|$ is bounded, $\bbE[T] < \infty$, and for fixed $u\ge 0$ and
$t\ge u$ $g(t - u)$ is directly Riemann integrable, the first integral
$\int_0^\infty u g(t-u)\, F(\rd u)$ is directly Riemann integrable.
Similarly, the second integral $\int_0^\infty g(t-u)\, F(\rd u)$ is also
directly Riemann integrable.  Hence, $w_1(t)$ is directly Riemann integrable.
We then compute the integration on $w_1$:
\begin{align*}
  \int_0^s w_1(t)\, \rd t & = \bbE\bigg[\int_0^s (\lambda T - 1)g(t-T) \,1(T \le t) \, \rd t \bigg] \\
 & = \bbE\bigg[\int_0^{(s-T)^+} (\lambda T - 1)g(u)\,1(u \ge 0)\,\rd u\bigg]\\
 & = \bbE\bigg[\int_0^s (\lambda T - 1) g(u)\,\rd u \bigg] 
  - \bbE\bigg[\int_{(s-T)^+}^s (\lambda T - 1) g(u) \, \rd u \bigg]\\
 & = \bbE\bigg[\int_{(s-T)^+}^s (1 - \lambda T) g(u)\,\rd u \bigg]\\
 & = \int_0^s \bbE\Big[(1-\lambda T) 1(T> s-u) \Big] g(u) \, \rd u,
\end{align*}
which is finite as $s \to \infty$ if $\bbE(T^2) < \infty$ because $|g(u)|$ is
bounded by $C$. Further, it is not hard to see that this integral converges
to 0 as $s \to \infty$.

We next consider $w_2(t)$. Since
\begin{equation*}
  w_2(t) = C\, \bbE\big[(1 - \lambda T)\big(1-1(T > t)\big)\big]
= - C \,\bbE[(1 - \lambda T)1(T > t)],
\end{equation*}
is directly Riemann integrable because $\bbE(T^2) < \infty$, and
\begin{equation}
\label{eq:w3 1}
  \int_0^\infty w_2(t)\, \rd t = \lambda^2 C\,\bbE\bigg[\int_0^\infty (T-\bbE[T])1(T > t)\, \rd t \bigg] = \lambda^2 C\, \sigma_T^2. 
\end{equation}

Hence, $w_1(t) + w_2(t)$ is directly Riemann integrable, and therefore the key
renewal theorem and \eq{w3 1} yield
\begin{align}
\label{eq:Key 3}
  \lim_{t\to\infty} \bbE\bigg[(T_{N(t)} + t_{N(t)-1} - t) \sum_{\ell=1}^{N(t)-1} (1 - \lambda T_\ell)\bigg]
& = \lambda \bigg( \int_0^\infty w_1(t)\, \rd t + \int_0^\infty w_2(t)\, \rd t \bigg) \nonumber\\
 & = \lambda \int_0^\infty w_2(t)\, \rd t  = \lambda^2 C\,\sigma_T^2.
\end{align}
Recalling that $C =  \frac 12 \lambda \bbE[T^2]$, \eq{Key 2} follows.
This proves \lem{RM 1}.

\subsection*{Acknowledgements}%
We thank two referees for helpful comments and suggestions.
DJD's work was done as an Honorary Professorial Associate at the University
of Melbourne.  MM's work was partly supported by JSPS KAKENHI Grant Number
16H027860001.

%

\def\cprime{$'$} \def\cprime{$'$} \def\cprime{$'$} \def\cprime{$'$}
  \def\cprime{$'$} \def\cprime{$'$} \def\cprime{$'$}

\end{document}